\documentclass[12pt]{amsart}

\usepackage{geometry}
\usepackage{amsmath, amssymb, amsthm, amsfonts, amsxtra, mathrsfs, mathtools} 

\usepackage[intoc]{nomencl}
\makenomenclature

\usepackage{extarrows}
\usepackage{import}
\usepackage{color}
\usepackage[colorlinks=true, allcolors=purple]{hyperref}
\usepackage{stmaryrd}
\usepackage{tikz-cd}
\usepackage{tikz}
\usepackage{tikzsymbols}
\usepackage[shortlabels]{enumitem}
\usetikzlibrary{matrix}
\usepackage[ruled,vlined]{algorithm2e}

\usepackage[mathscr]{euscript}

\input xy
\xyoption{all}

\numberwithin{equation}{section}

\theoremstyle{definition}
\newtheorem{thm}{Theorem}[section]
\newtheorem{lemma}[thm]{Lemma}

\newtheorem{corollary}[thm]{Corollary}

\newtheorem{definition}[thm]{Definition}
\newtheorem{solution}{Solution}

\newtheorem{remark}[thm]{Remark}
\newtheorem*{digression*}{Digression}

\newtheorem*{lemma*}{Lemma}

\newcommand\Aut{\operatorname{Aut}}

\newcommand{\GL}{\textup{GL}}

\newcommand\Hom{\textup{Hom}}

\usepackage{CJKutf8} 

\def\multiset#1#2{\ensuremath{\left(\kern-.3em\left(\genfrac{}{}{0pt}{}{#1}{#2}\right)\kern-.3em\right)}}

\newlength\mylen
\newlist{mycases}{enumerate}{1}
\setlist[mycases,1]{label=\textbf{Case~\arabic*.}, 
  labelwidth=\dimexpr-\mylen-\labelsep\relax,leftmargin=0pt,align=right}

\newcommand{\repeatable}[2]{%
  \label{#1}\global\@namedef{repeatable@#1}{#2}#2
}
\renewcommand{\repeat}[1]{%
  \@ifundefined{repeatable@#1}{NOT FOUND}{$\@nameuse{repeatable@#1}$}%
  ~\ref{#1}}
\makeatother

\title[Twisted Alexander Polynomials and Ribbon Homology Cobordisms]{Twisted
Alexander Polynomials and Fibered Classes in Ribbon Homology Cobordisms}
\author{Brian Sun}\address{Brian Sun}\email{bsun@berkeley.edu}

\begin{document}
\begin{abstract}
    Let $Y_-$ and $Y_+$ be two compact 3-manifolds with empty or toroidal
    boundary. A 4-dimensional ribbon homology cobordism is a homologically
    trivial cobordism built with 1-handles and 2-handles. In this note,
    following the work of Friedl and collaborators, we apply twisted Alexander
    polynomials to show that the fibered classes of $Y_+$ map to those of $Y_-$.
\end{abstract}

\maketitle

\allowdisplaybreaks

\section{Introduction}

We begin by introducing 4-dimensional ribbon homology cobordisms, first studied
in great generality in \cite{DAEMI2022108580}. Let $Y_-$ and $Y_+$ be two
3-manifolds. Unless otherwise stated, all 3-manifolds will be connected,
compact, and orientable, with either empty or toroidal boundary.

A \textit{ribbon cobordism} is a cobordism $W$ from $Y_-$ to $Y_+$
that can be built using only 1-handles and 2-handles. If $Y_-$ has boundary, we
only consider handle attachments to the interior, so the boundary is maintained
all the way. An \textit{$R$-homology cobordism} is a cobordism $W$ where
$H_*(W,Y_-;R)=H_*(W,Y_+;R)=0$. By a \textit{ribbon homology cobordism} we will
mean a ribbon $\mathbb{Q}$-homology cobordism unless specified.

One natural example of a ribbon homology cobordism (in fact, a
$\mathbb{Z}$-homology cobordism) is that of a ribbon concordance, as introduced
in \cite{GordonRibbonConcordance}. Let $L_0,L_1\subset S^3$ be two oriented
$m$-component links. A concordance between $L_0$ and $L_1$ is an embedding
$e\colon \bigsqcup^m S^1\times[0,1]\rightarrow S^3\times[0,1]$ such that
$e(S^1\times \{0\})= L_0\times\{0\}$ and $e(S^1\times\{1\})= L_1\times\{1\}$. We
denote $C\coloneqq \operatorname{im}(e)$ and identify $L_0$ and $L_1$ with their
images in $S^3\times[0,1]$. By a small perturbation we may assume that the
projection $p\colon S^3\times [0,1]\rightarrow [0,1]$ is a Morse function. We
say $C$ is a \textit{ribbon concordance} from $L_1$ to $L_0$ if $p|_C$ has no
index 2 critical points (no local maxima), and in this case we can write
$L_1\geq L_0$. The exterior $(S^3\times I)\setminus \mathcal{N}(C)$ defines a
ribbon (homology) cobordism between link exteriors $X_0$ and $X_1$, which can be
seen by imagining filling up $S^3\times I$ with water and analyzing how the
topology changes as the surface passes critical points. 

An important property of homology cobordisms for our purposes is that it gives
natural identifications $H_*(Y_-;\mathbb{Q})\cong H_*(W;\mathbb{Q})\cong
H_*(Y_+;\mathbb{Q})$. We thus get natural identifications $H^1(Y_-;\mathbb{Q})
\cong H^1(W;\mathbb{Q})\cong H^1(Y_+;\mathbb{Q})$ via the universal coefficient
theorem. This lets us compare classes in $H^1$, and one natural question to ask
is how the fibered classes (or perhaps more precisely, rays) of the Thurston
norm ball correspond under the identification of $H^1(Y_-;\mathbb{Q})$ and
$H^1(Y_+;\mathbb{Q})$. We have the following result:

\begin{thm}
\label{thm:fibered}
    Let $Y_-$ and $Y_+$ be 3-manifolds with a ribbon homology cobordism $W$ from
    $Y_-$ to $Y_+$. If $\phi_+\in H^1(Y_+;\mathbb{Q})$ is a fibered class, then
    the associated class in $\phi_-\in H^1(Y_-;\mathbb{Q})$ is also fibered.
\end{thm}

Hence each fibered open cone of $Y_+$ is sent to a fibered open cone of $Y_-$.
Specializing to ribbon concordances, we have the following result, first shown
in the case of knots \cite{SILVER199299,Kochloukova}.

\begin{corollary}
    Let $L_0$ and $L_1$ be links with $L_1\geq L_0$ with $L_1$ nontrivial. If
    $L_1$ is fibered, then so is $L_0$.
\end{corollary}
\begin{proof}
    Each meridian of $L_0$ gets paired with a meridian of $L_1$, and so the
    element $(1,1,\ldots,1)\in H^1(X_0;\mathbb{Q})$ is identified with the
    analogous element in $H^1(X_1;\mathbb{Q})$.
\end{proof}

We will prove the theorem using the theory of twisted Alexander polynomials
that has been developed in \cite{FRIEDL2006929,
FriedlHomotopyAlexander,FriedlVanishing,
friedl2012thurstonnormtwistedalexander}. The hope is that some of these ideas
can complement Floer-theoretic techniques in working with ribbon homology
cobordisms.

\section{Background}

\subsection{Twisted Alexander polynomials} Two good references for this section
are \cite{friedl2010surveytwistedalexanderpolynomials} and \cite{MR1809561}. Let $X$ be
a space with universal cover $\tilde X$. The group $\pi_1(X)$ acts on the
singular chain complex $C_*(\tilde X)$ from the left as deck transformations. We
convert this into a right action in the natural way. Then if $A$ is a left
$\mathbb{Z}[\pi_1(X)]$-module given by $\beta\colon \pi_1(X)\rightarrow
\Aut(A)$, we define the \textit{$i$th twisted homology group} as 
\[ 
    H_i(X;A)\coloneqq H_i(C_*(\tilde X;\mathbb{Z})\otimes_{\mathbb{Z}[\pi_1(X)]}
    A).
\]
We consider the following special case. Let $M$ be a compact 3-manifold and
$\alpha\colon \pi_1(M)\rightarrow \GL(k,\mathbb{C})$ a representation. Further
let $\phi\colon \pi_1(M)\rightarrow \mathbb{Z}$ be a nontrivial homomorphism,
i.e., a nonzero element of $H^1(M;\mathbb{Z})$. Then we can form the tensor
representation
\[ 
    \alpha\otimes\phi\colon \pi_1(M)\rightarrow \GL(k,\mathbb{C}[t^{\pm 1}])
\]
where we map $g$ to the product $\alpha(g)\cdot \phi(g)$, interpreting $\phi(g)$
as a scaling. Then we will write 
\[ 
    H_i^{\alpha\otimes \phi}(M;\mathbb{C}[t^{\pm 1}]^k) \coloneqq 
    H_i(C_*(\tilde M;\mathbb{Z})\otimes_{\mathbb{Z}[\pi_1(M)]}\mathbb{C}[t^{\pm
    1}]^k).
\]
The \textit{$i$th twisted Alexander polynomial} of $(M,\phi,\alpha)$ is the
order of $H_i^{\alpha\otimes\phi}(M;\mathbb{C}[t^{\pm 1}]^k)$, denoted by
$\Delta_{M,\phi,i}^\alpha\in \mathbb{C}[t^{\pm 1}]$. It is well-defined up to
a multiplication by a unit in $\mathbb{C}[t^{\pm 1}]$.

We now cite two results showing how twisted Alexander polynomials can obstruct
fibering and detect non-fibering:

\begin{thm}[{\cite[Theorem~1.3]{FRIEDL2006929}}]
\label{thm:obstruct-fiber}
    Let $M$ be a 3-manifold that is not $S^1\times D^2$ or $S^1\times S^2$, and
    let $\phi\in H^1(M;\mathbb{Z})$ be a nontrivial class. If $\phi$ is a
    fibered class, then for any representation $\alpha\colon \pi_1(M)\rightarrow
    \GL(k,R)$ with $R$ a Noetherian UFD, $\Delta_{M,\phi,1}^\alpha$ is monic.
\end{thm}

\begin{remark}
    In fact, it is also proven in \cite{FRIEDL2006929} that in the fibered case
    the twisted Alexander polynomial detects the Thurston norm.
\end{remark}

The exclusion of $S^1\times D^2$ and $S^1\times S^2$ is due to the fact that the
proof requires $x_M(\phi)=\chi_-(S)$ with $S$ dual to $\phi$. We remark that a
converse of the above theorem has been shown \cite{FriedlVidussi2011}.

\begin{thm}[{\cite[Theorem~1.1]{FriedlVanishing}}]
\label{thm:nonfibered-vanishing}
    Let $M$ be a 3-manifold and $\phi\in H^1(M;\mathbb{Z})$ a nontrivial class.
    If $\phi$ is nonfibered, then there exists an epimorphism $\alpha\colon
    \pi_1(M)\rightarrow G$ onto a finite group $G$ such that
    \[ 
        \Delta_{M,\phi,1}^\alpha=0.
    \]
\end{thm}
Here we are considering the left regular representation $\pi_1(M)\rightarrow
G\rightarrow \Aut_{\mathbb{Z}}(\mathbb{Z}[G])$. We identify this latter set with
$\GL(|G|,\mathbb{Z})$ and then construct the twisted homology group with
coefficients in $\mathbb{Z}^{|G|}[t^{\pm 1}]$, from which we obtain
$\Delta_{M,\phi,1}^\alpha$.

\subsection{Homotopical properties of ribbon homology cobordisms}

Following Gordon \cite{GordonRibbonConcordance}, one can show that if
we have a ribbon homology cobordism $W$ from $Y_-$ to $Y_+$, then the inclusion
$Y_+\hookrightarrow W$ induces a surjection $\iota_+\colon \pi_1(Y_+)\rightarrow
\pi_1(W)$, and the inclusion $Y_-\hookrightarrow W$ induces an injection
$\iota_-\colon \pi_1(Y_-)\rightarrow \pi_1(W)$. The argument is also spelled out
in \cite[Proposition 2.1]{DAEMI2022108580}.

Specializing to links, we can take this to define a topological version of
ribbon concordance: the links $L_1$ and $L_0$ are \textit{homotopy ribbon
concordant} if $L_1$ and $L_0$ are concordant by a locally flatly embedded
concordance $C$ such that $\iota_1\colon \pi_1(X_1)\rightarrow \pi_1((S^3\times
I)\setminus \mathcal{N}(C))$ is surjective and $\iota_0\colon
\pi_1(X_0)\rightarrow \pi_1((S^3\times I)\setminus \mathcal{N}(C))$ is
injective. The relation between twisted Alexander polynomials and homotopy
ribbon concordance is given by the following:

\begin{thm}[{\cite[Theorem 1.3]{FriedlHomotopyAlexander}}]
\label{thm:homotopy-ribbon-alexander}
    Let $Y$ be the exterior of a homotopy ribbon concordance from $K_1$ to $K_0$
    and $\alpha\colon \pi_1(Y)\rightarrow \GL(k,R)$ a representation for $R$ a
    Noetherian UFD. Then
    \[ 
        \Delta_{X_0,\phi\circ\iota_0,1}^{\alpha\circ\iota_0} \mid
        \Delta_{X_1,\phi\circ\iota_1,1}^{\alpha\circ\iota_1}
    \]
    where $\phi\in H^1(Y;\mathbb{Z})\cong \Hom(\pi_1(Y),\mathbb{Z})$ is a
    generator that takes a meridian to $1$.
\end{thm}

Though the above result is stated in terms of homotopy ribbon concordances, the
proof is purely homological algebra and depends only on the fact that we have a
surjection $\pi_1(X_1)\rightarrow \pi_1(Y)$. The exact same argument lets us
extend to the following situation.

\begin{corollary}
\label{cor:divisibility}
    Let $W$ be a ribbon homology cobordism from $Y_-$ to $Y_+$. Suppose
    $\phi_-\in H^1(Y_-;\mathbb{Z})$ and $\phi_+\in H^1(Y_+;\mathbb{Z})$ are
    associated under $W$. If $\alpha\colon \pi_1(W)\rightarrow \GL(k,R)$ is a
    representation with $R$ a Noetherian UFD, then
    \[ 
        \Delta_{Y_-,\phi_-,1}^{\alpha\circ\iota_-} \mid
        \Delta_{Y_+,\phi_+,1}^{\alpha\circ\iota_+}.
    \]
\end{corollary}

\section{Proof of main theorem}

Our identifications of (co)homology are over $\mathbb{Q}$, but we can always
take multiple and to ensure that the two classes we are comparing in $Y_-$ and
$Y_+$ are both integral. Since our comparisons are invariant under scaling, we
will always assume we are working with integral classes. We shall prove the
contrapositive.

We first restate an argument of \cite{GordonRibbonConcordance}
that lets us extend a representation of $\pi_1(Y_-)$ to $\pi_1(W)$, from which
we also get a representation of $\pi_1(Y_+)$.

\begin{lemma}
\label{lem:extend}
    Let $\beta\colon \pi_1(Y_-)\rightarrow G$ be a representation to a compact
    connected Lie group $G$. Then $\beta$ can be extended to a representation
    $\alpha\colon \pi_1(W) \rightarrow G$ that restricts to $\beta$.
\end{lemma}
\begin{proof}
    In a ribbon homology cobordism we have $W=Y_-\times I\cup
    \text{1-handles}\cup\text{2-handles}$. This means we can write
    \[ 
        \pi_1(W)\cong \frac{\pi_1(Y_-)*F}{\langle\!\langle r_1,\ldots,
        r_n\rangle\!\rangle}.
    \]
    Since $H_*(Y_-;\mathbb{Q})\rightarrow H_*(W;\mathbb{Q})$ is
    an isomorphism, there are the same number of $1$-handles (generators of $F$)
    and $2$-handles (relations), with the $2$-handles cancelling out the
    $1$-handles homologically. This implies that the $r_i$ satisfy the
    conditions of \cite[Theorem 1(ii)]{GerstenhaberRothaus}, from which we
    immediately obtain $\alpha$.
\end{proof}

\begin{proof}[Proof of Theorem~\ref{thm:fibered}]
    We begin with the case where $Y_+$ is $S^1\times D^2$ or $S^1\times S^2$.
    From the facts that $\pi_1(Y_+)$ surjects onto $\pi_1(W)$, $\pi_1(Y_-)$
    injects into $\pi_1(W)$, and that the rational first homologies of $Y_+$,
    $Y_1$, and $W$ are identified (all $\mathbb{Q}$ in this case), we must have
    $\pi_1(Y_-)\cong\mathbb{Z}$. By the Poincar\'e conjecture, $Y_-$ is either
    one of $S^1\times D^2$ or $S^1\times S^2$. But $S^1\times D^2$ has torus
    boundary while $S^1\times S^2$ is closed, so in either case it must be that
    $Y_-$ is identified with $Y_+$. Since the Thurston norm ball is
    1-dimensional in both cases, we are done.

    Let $\phi_-\in H^1(Y_-;\mathbb{Z})$ be a nonfibered class, and
    $\phi_+\in H^1(Y_+;\mathbb{Z})$ the associated class under $W$. By
    Theorem~\ref{thm:nonfibered-vanishing}, there is a surjection $\beta\colon
    \pi_1(Y_-)\rightarrow G$ to some finite group $G$ such that
    $\Delta_{Y_-,\phi_-,1}^\beta=0$. We can further compose with the map
    $G\rightarrow U(|G|)$ induced by the left regular representation to obtain a
    unitary representation we still call $\beta$. Then by construction $\beta$
    is really just $\alpha$ with $\mathbb{C}$ coefficients, so we get
    $\Delta_{Y_-,\phi_-,1}^\beta=0$.

    From Lemma~\ref{lem:extend}, we can extend the representation $\beta$ on
    $\pi_1(Y_-)$ along $\iota_-$ to a representation $\alpha$ on $\pi_1(W)$ with
    $\beta=\alpha\circ\iota_-$. As $\phi_-$ extends to $\phi$ along $\iota_-$ as
    well, we obtain a representation $\alpha\otimes\phi\colon
    \pi_1(W)\rightarrow U(|G|)$ that restricts to $\beta\otimes \phi_-$ on
    $\pi_1(Y_-)$. An application of Corollary~\ref{cor:divisibility} 
    immediately yields $\Delta_{Y_+,\phi_+,1}^{\alpha\circ\iota_+}=0$. By
    Theorem~\ref{thm:obstruct-fiber} the class $\phi_+$ must be nonfibered as
    the zero polynomial is not monic.
\end{proof}

\begin{remark}
    It would be interesting to see if we could prove that the Thurston norm was
    monotonic under ribbon concordance using twisted Alexander polynomials, as
    seen in \cite{DAEMI2022108580}. The corresponding result for knot genus was
    first proven in \cite{Zemke}. The primary issue with using these methods
    described here is that though there is a representation on $\pi_1(Y_-)$ that
    detects the Thurston norm as seen in
    \cite{friedl2012thurstonnormtwistedalexander}, what we end up requiring is
    $\deg(\Delta_{Y_-,1})\leq \deg(\Delta_{Y_+,1})$. Though we have divisibility
    from Corollary~\ref{cor:divisibility}, there is no guarantee that the right
    hand side is nonzero. This would require an analysis of the representation
    variety and whether or not the component we land on with respect to $Y_+$
    has a representation with nonzero first Alexander polynomial. 
\end{remark}

\section*{Acknowledgements}

I would like to thank Ian Agol for many helpful discussions and feedback. I'd
like to thank Qiuyu Ren for comments and a strengthening of
Theorem~\ref{thm:fibered}. I would also like to thank Jianru Duan for pointing
out an error in the first draft. I was partially supported by the Simons
Investigator grant \#376200. I'm also grateful to SL Math for providing
facilities to discuss and think about this work.

\bibliography{bibliography}{}

\newcommand{\etalchar}[1]{$^{#1}$}
\providecommand{\bysame}{\leavevmode\hbox to3em{\hrulefill}\thinspace}
\providecommand{\MR}{\relax\ifhmode\unskip\space\fi MR }
\providecommand{\MRhref}[2]{%
  \href{http://www.ams.org/mathscinet-getitem?mr=#1}{#2}
}
\providecommand{\href}[2]{#2}
\begin{thebibliography}{DLVVW22}

\bibitem[DLVVW22]{DAEMI2022108580}
Aliakbar Daemi, Tye Lidman, David~Shea Vela-Vick, and C.-M.~Michael Wong, \emph{Ribbon homology cobordisms}, Advances in Mathematics \textbf{408} (2022), 108580.

\bibitem[FK06]{FRIEDL2006929}
Stefan Friedl and Taehee Kim, \emph{{The Thurston norm, fibered manifolds and twisted Alexander polynomials}}, Topology \textbf{45} (2006), no.~6, 929--953.

\bibitem[FKL{\etalchar{+}}22]{FriedlHomotopyAlexander}
Stefan Friedl, Takahiro Kitayama, Lukas Lewark, Matthias Nagel, and Mark Powell, \emph{Homotopy ribbon concordance, {B}lanchfield pairings, and twisted {A}lexander polynomials}, Canad. J. Math. \textbf{74} (2022), no.~4, 1137--1176. \MR{4464583}

\bibitem[FV10]{friedl2010surveytwistedalexanderpolynomials}
Stefan Friedl and Stefano Vidussi, \emph{{A survey of twisted Alexander polynomials}}, 2010.

\bibitem[FV11]{FriedlVidussi2011}
Stefan Friedl and Stefano Vidussi, \emph{Twisted {A}lexander polynomials detect fibered 3-manifolds}, Annals of Mathematics \textbf{173} (2011), no.~3, 1587--1643. \MR{2800721}

\bibitem[FV13]{FriedlVanishing}
\bysame, \emph{A vanishing theorem for twisted alexander polynomials with applications to symplectic 4-manifolds}, Journal of the European Mathematical Society \textbf{15} (2013), 2027--2041.

\bibitem[FV15]{friedl2012thurstonnormtwistedalexander}
Stefan Friedl and Stefano Vidussi, \emph{{The Thurston norm and twisted Alexander polynomials}}, Journal für die reine und angewandte Mathematik (Crelles Journal) \textbf{2015} (2015), no.~707, 87--102.

\bibitem[Gor81]{GordonRibbonConcordance}
C.~McA. Gordon, \emph{Ribbon concordance of knots in the {$3$}-sphere}, Math. Ann. \textbf{257} (1981), no.~2, 157--170. \MR{634459}

\bibitem[GR62]{GerstenhaberRothaus}
Murray Gerstenhaber and Oscar~S. Rothaus, \emph{The solution of sets of equations in groups}, Proceedings of the National Academy of Sciences of the United States of America \textbf{48} (1962), no.~9, 1531--1533.

\bibitem[Koc06]{Kochloukova}
Dessislava Kochloukova, \emph{Some novikov rings that are von neumann finite and knot-like groups}, Commentarii Mathematici Helvetici \textbf{81} (2006), 931--943.

\bibitem[Sil92]{SILVER199299}
D.S. Silver, \emph{On knot-like groups and ribbon concordance}, Journal of Pure and Applied Algebra \textbf{82} (1992), no.~1, 99--105.

\bibitem[Tur01]{MR1809561}
Vladimir Turaev, \emph{Introduction to combinatorial torsions}, Lectures in Mathematics ETH Z\"urich, Birkh\"auser Verlag, Basel, 2001, Notes taken by Felix Schlenk. \MR{1809561}

\bibitem[Zem19]{Zemke}
Ian Zemke, \emph{Knot {F}loer homology obstructs ribbon concordance}, Ann. of Math. (2) \textbf{190} (2019), no.~3, 931--947. \MR{4024565}

\end{thebibliography}
\bibliographystyle{amsalpha}

\end{document}